\documentclass[11pt,a4paper]{article}
\usepackage[utf8]{inputenc}
\usepackage{amsmath,amssymb,amsthm,color,graphicx,fancybox}
\usepackage{tcolorbox}
\usepackage{bclogo}
\usepackage{mdframed}
\usepackage{wasysym}
\usepackage[english]{babel}
\usepackage{subfigure}
\usepackage{subfigmat}
\usepackage{float}

\newcommand{\R}{\mathbb R}

\newcommand{\eps}{\epsilon}

\newtheorem{theorem}{Theorem}[section]
\newtheorem{lemma}{Lemma}[section]

\newtheorem{cor}{Corollary}[section]

\newtheorem{remark}{Remark}[section]
\newtheorem{defi}{Definition}[section]
\newtheorem{acknowledgment*}{Acknowledgment}

\newcommand{\be}{\begin{equation}}
\newcommand{\ee}{\end{equation}}

\title{\Large \textbf{Happy family of stable marriages}}

\author{\textsc{Gershon Wolansky}\footnote{Department of Mathematics, Technion, Israel Inst. of Technology}}

\begin{document}\huge
\begin{center}{\bf Self-perimeters of convex sets}\end{center}
\normalsize
\begin{center}
 Gershon Wolansky, \\ \ Technion, Israel Inst. of Technology, \ Haifa, Israel
\end{center}

\begin{abstract}
 This paper introduces a natural definition for the volume of the unit ball in $n$-dimensional normed spaces $\mathbb{R}^n$. This definition preserves the Euclidean relation $P(B)/V(B)=n$ between the perimeter and the volume of the unit ball $B$ in $\mathbb{R}^n$. We show that this volume definition is invariant under origin-preserving affine transformations. For $n=2$, we derive an explicit integral formula for the self-perimeter of the unit ball and extend it to non-centrally symmetric sets. The construction is extended to $\mathbb{R}^n$ via a recursive integration over the boundary, utilizing $(n-1)$-dimensional volumes of planar intersections, and we discuss its invariance under affine transformations. Finally, we pose and discuss an Alexandrov-type problem for the associated surface measure, providing perturbative solutions in the 2D case. In particular, we prove that, generically, any perturbation of the surface measure of the Euclidean 2D disk yields a 4-fold symmetric convex set in the leading order.
\end{abstract}

\section{Introduction}

The concept of volume and surface area in finite-dimensional normed spaces (Minkowski spaces) lacks a single, universally accepted definition, unlike in standard Euclidean geometry. Over the years, various approaches have been proposed, each tailored to preserve specific geometric, analytic, or integral properties of the space \cite{alvarez1998volumes, thompson1996minkowski}. Among the most prominent classical frameworks are the Busemann and the Holmes-Thompson definitions. The Busemann definition \cite{busemann1950geometry} normalizes the measure such that the Minkowski unit ball has the same volume as the Euclidean unit ball. The corresponding affine invariant surface area is then obtained by integrating the Euclidean surface element, normalized by the volume of the central cross-section of the unit ball parallel to the tangent plane. Alternatively, the Holmes-Thompson definition \cite{holmes1979n}, originating from symplectic geometry, measures the volume of a set by multiplying its Lebesgue measure with the Lebesgue measure of the polar body in the dual space, normalized by the Euclidean unit ball volume. This approach inherently satisfies the Blaschke-Santaló inequalities \cite{schneider2014convex} and is completely coordinate-independent. However, these classical definitions do not inherently prioritize the preservation of the fundamental Euclidean relationship between the perimeter (or surface area) $P_n$ and the volume $V_n$ of the unit ball, namely $P_n = n V_n$. In this paper, we attempt to find a natural definition of volume and perimeter in normed spaces $\mathbb{R}^n$ based precisely on the preservation of this Euclidean ratio $P_B(B)/V(B)=n$ for the unit ball $B \subset \mathbb{R}^n$. To achieve this, we introduce an intrinsic, recursive approach, starting with $n=2$ and using the definition of \textit{self-perimeter}.

The concept of self-perimeter is a central object of study in asymmetric Minkowski geometry and the theory of convex distance functions. A fundamental extremal problem in this field involves bounding the extremal self-perimeter of convex bodies in dimension two. The historical starting point is Golab's theorem \cite{golab1932}, which established that for centrally symmetric convex bodies, the self-perimeter is bounded between six and eight (with the maximum attained by a parallelogram, the minimum by the perfect hexagon, while the Euclidean $2\pi$ lies in between). Extending this problem to the space of asymmetric convex bodies expands the upper bound. The precise result states that the minimum self-perimeter of any planar convex body (calculated with respect to the optimal interior reference point) is bounded between six and nine. The lower bound of six is attained if and only if the body is an affinely regular hexagon. The absolute maximum of nine is attained if and only if the convex body is a triangle, a property formally proven by Grünbaum \cite{grunbaum1966}. These results are further supported by the strict convexity of the self-perimeter with respect to the location of the interior reference point \cite{makeev2003}. For regular $k$-gons, as well as for the family of planar convex bodies possessing $2\pi/k$ rotational symmetry, the self-perimeter exhibits periodic behavior depending on the remainder of $k$ modulo four \cite{ghandehari2019, shcherba2007}. The literature demonstrates that when $k \equiv 0 \pmod 4$, the regular polygon constitutes a maximum (upper bound) for the family of shapes with the same degree of symmetry. Conversely, when $k \equiv 2 \pmod 4$, the regular polygon serves as a minimum \cite{martini2001}. This modular partition dictates the nature of the convergence of the self-perimeter to the limit of $2\pi$ (the value for a circle).

\subsection*{Extension to higher dimensions}
In this paper, we introduce a twist on Busemann's definition \cite{busemann1950geometry} of an affine invariant perimeter of a convex set $K$ containing the origin in its interior:
\be\label{Bus}P_{Bus}(K)=\int_{\partial K} \sigma^{(n-1)}_K(\nu_x) d{\cal H}_{n-1}(x) \ee
where $\nu_x$ is the unit normal to $\partial K$ at $x$.
\be\label{sigmaK}\sigma_K^{(n-1)}(\nu_x)\equiv \frac{\omega^{(n-1)}}{ {\cal H}_{n-1} (K\cap \Sigma^{n-1}(\nu_x))}\ee
$\Sigma^{n-1}(\nu_x)$ is the $(n-1)$-hyperplane perpendicular to the normal $\nu_x\in\partial K$ and $$\omega^{(k)}=
\frac{\pi^{k/2}}{\Gamma \left(\frac{k}{2}+1\right)}$$
is the volume of the unit $k$-ball in Euclidean space. Considering a centrally symmetric convex set $B\in \R^n$, we are looking for a sensible definition of a perimeter of a set $K\subset \R^n$ corresponding to the norm $\|\cdot\|_B$ induced by $B$ as a unit ball. Our definition is inspired by the affine invariant definition of Busemann (\ref{Bus}) in dimension $2$. In fact, if we consider $K$ to be the unit ball $B$ in $\R^2$, then equation (\ref{Bus}) is the actual perimeter of $B$ in the metric $\|\cdot\|_B$ (see Sec. \ref{sec3}), called \textit{the self-perimeter} of $B$.

In our approach, we replace $\omega^{(n-1)}$ in  (\ref{sigmaK}) by a recursive definition of $\omega^{(n-1)}_B(\nu)$, which is the {\em self-volume} of the unit $n-1$ dimensional ball  $B\cap \Sigma^{n-1}(\nu)$ {\em congruent with the induced norm $\|\cdot\|_{B\cap \Sigma^{n-1}(\nu)}$ on $\Sigma^{n-1}(\nu)$. } 
 See Definition \ref{def4.1} below. 

Thus, the perimeter of a set $K\subset\R^n$ consistent with the norm $\|\cdot\|_B$ is:
$$P_B(K)=\int_{\partial K} \frac{\omega^{(n-1)}_B(\nu_x)}{ {\cal H}_{n-1} ( B\cap \Sigma^{n-1}(\nu_x))} d{\cal H}_{n-1}(x) \ , $$
while the self perimeter  $P(B)$ of $B\subset \R^n$ itself  and its self-volume $\omega^{(n)}_B$  are
$$P(B)=\int_{\partial B} \frac{\omega^{(n-1)}_B(\nu_x)}{ {\cal H}_{n-1} ( B\cap \Sigma^{n-1}(\nu_x))} d{\cal H}_{n-1}(x) =n\omega^{(n)}_B \ . $$
This definition coincides with Busemann's definition for $n=2$ since $\omega^{(1)}=2$ is independent of any direction $\nu\in\R^2$. However, it replaces Busemann's preservation of the Euclidean volume of the unit ball with the preservation of the Euclidean ratio $P(B)/\omega_B^{(n)}=n$ of perimeter/volume of the unit ball in $n$ dimensions (Definition \ref{def4.1}).


As we will show, the exact cancellation of the Jacobian in this formulation ensures that this definition of volume is strictly preserved under origin-preserving affine transformations. We extend the definition of $\omega_B^{(n)}$ to non-centrally symmetric convex bodies $B$, calculate explicitly the self-volume/perimeter of the  $n$-simplex $\Delta^n$, and prove the  strict convexity
of   $p\mapsto \omega^{(n)}_{B-p}$. We pose the following:
\vskip .2in\noindent 
{\bf Conjecture 1.} 
{\it The function $p\in Int(B) \mapsto \omega^{(n)}_{B-p}$ is strictly convex for any any convex set $B\subset\R^n$. In particular there exists a unique minimizer $p^*\in int(B)$.  
}

Conjecture 1 is valid for any convex $B\subset \R^2$ (cf. Theorem \ref{th3.4} below).  In Section \ref{sec4.3}
  we show that this conjecture is valid  of any convex body in $\R^n$ composed of Cartesian products of simplices, 1D segments, and 2D convex sets.  At the end of Section \ref{sec40} we pose the following  (see Section \ref{subsec4.3}):

\vskip .2in\noindent 
{\bf Conjecture 2.} 
{\it 
\begin{itemize}
    \item The CCS (convex centrally symmetric) body in $\R^n$ of maximal self-volume is the $n$-cube $B=[-1,1]^n$,  $\omega^{(n)}_B=2^n$.
    \item For any convex set $B\subset\R^n$ containing the origin, $\omega^{(n)}_B\leq \frac{(n+1)^n}{n!} = \omega^{(n)}_{\Delta^n}$, where $\Delta^n$ is the $n-$simplex and $0\in\Delta^n$ is its centeroid. 
    \item If $n=2k$, then the CCS body of minimal self-volume is the Cartesian product of $k$ hexagons ($\omega^{(n)}_B=3^k$). If $n=2k+1$, it is the Cartesian product of $k$ hexagons and a 1D interval ($\omega^{(n)}_B=2\times 3^k$).
\end{itemize}}

\subsection*{Alexandrov-type Problem for Self-Surface Measure}

Beyond the foundational task of defining volume and surface area, this framework allows us to pose a novel inverse problem. In classical convex geometry, the Minkowski and Alexandrov problems seek to reconstruct a convex body from its standard Euclidean surface area measure or integral Gauss curvature, respectively \cite{schneider2014convex}. Within the context of our intrinsic definitions, we introduce an Alexandrov-type problem formulated with respect to the self-surface measure of a centrally symmetric convex set \cite{alvarez1998volumes, thompson1996minkowski}. Given that our perimeter and volume are defined recursively by integrating over the essential boundary—normalized by the self-volume of $(n-1)$-dimensional cross-sections—the natural inverse question arises: Given a positive measure $\mu$ on the Euclidean sphere $S^{n-1}$, does there exist a centrally symmetric convex set $B \subset \mathbb{R}^n$ whose self-surface measure is exactly $\mu$?

In Section \ref{sec6} we explore this problem in the two-dimensional case, providing perturbative solutions and identifying necessary conditions for its solvability around the Euclidean disc. In particular, we obtain that an $\mathcal{O}(\eps)$ perturbation of the uniform measure on the unit Euclidean circle yields an $\mathcal{O}(\eps^{1/2})$ perturbation of this circle, which depends \textit{only on the log 4-fold symmetric} part of the measure (c.f., Theorem \ref{th6.1}).

\section{Fundamental definitions}
Let $B\subset \R^n$ be a convex, centrally symmetric set (CCS). The norm $\|\cdot\|_B$ is defined in the usual way as:
$$ \|v\|_B=\sup\{ t>0; \ \ v/t\not\in B\}\ \ . $$
Note that $B$ is the unit ball in $\R^n$ with respect to $\|\cdot\|_B$.
\begin{defi}\label{def1}
 Let $K$ be a starlike set in $\R^n$ with respect to the origin. For $\theta\in S^{(n-1)}$ set:
 $$ r_K(\theta)= \sup \{t>0; t\theta\in K\} \ $$
 the \textit{radius function} of $K$.
\end{defi}
If $K$ is centrally symmetric (CCS), then $r_K(\theta)=r_K(-\theta)$.
\begin{defi} 
 The \textit{support function} $h_B(x):= \sup_{z\in B}\langle x, z \rangle$. The polar set $B^*:=\{x; h_B(x)\leq 1 \}$. If $B$ is CCS, so is $B^*$.
\end{defi}
\begin{lemma}[Schneider \cite{schneider2014convex}]
 For $\theta\in S^{n-1}$:
 $$ h_B(\theta) = \frac{1}{r_{B^*}(\theta)} \ \ ; \ \ r_B(\theta)= \frac{1}{h_{B^*}(\theta)} \ . $$
\end{lemma}

\section{Self-perimeter in the 2-plane}\label{sec3}
If $n=2$, there is a natural notion of a perimeter $P_B$ of $B$ as measured with respect to this norm. Indeed, we can use the norm $\|\cdot\|_B$ as a natural definition of length of an arc.

In the case $n=2$, we can introduce an explicit formulation for the perimeter:
\be\label{per2} P(B)= \int_{\partial B} \frac{ds}{r_B(\vec{\tau}(s))} \equiv \int_{\partial B}h_{B^*}\left(\vec{\tau}(s)\right)ds \ee
where $ds$ is the Euclidean length of the boundary of $B$, and $\vec{\tau}(s)$ is the tangent vector to the boundary at $s\in\partial B$. An explicit expression in terms of an integral on the Euclidean circle $\mathbb{S}^1$ is:

\begin{lemma}
 The self-perimeter of a convex $B\ni\{0\}$ in $\R^2$ is:
 \be\label{1}P(B)=\int_0^{2\pi} \frac{\sqrt{r^2_B(\phi)+ (r^{'}_B(\phi))^2}d\phi}{r_B\left( \phi + \arctan\left(\frac{r_B(\phi)}{r_B^{'}(\phi)}\right)\right)}\ee
 where $r_B(\phi):= r_B(\cos(\phi), \sin(\phi))$.
\end{lemma}
\begin{proof}
 The vector $\vec{r}= (r\cos(\phi), r\sin(\phi))$. Its tangent is $d\vec{r}/d\phi= (r^{'}\cos(\phi)-r\sin(\phi), r^{'}\sin(\phi) + r\cos(\phi))$. The slope of this direction is:
 $$\frac{\tan(\phi) + r/r^{'}}{1-(r/ r^{'})\tan(\phi)}= \tan\left( \phi+ \arctan(r/r^{'})\right) \ . $$
\end{proof}

\begin{theorem}[\cite{golab1932, martini2001}]
 The self-perimeter of CCS is always between 4 and 6. 
\end{theorem}
\begin{cor}
 If $r_B(\phi)$ is the radius function of a CCS in $\R^2$, then (\ref{1}) takes values between 6 and 8. 
\end{cor}
Note that for $B\in CCS(\R^2)$, $r_B$ is a $\pi$-periodic function. 

The definition (\ref{1}) can easily be extended to convex, non-centrally symmetric sets. Let $K$ be such a set and $p\in Int(K)$.
\begin{remark}\label{rem3.1}
 If $B$ is convex but not centrally symmetric, the perimeter $P$ may depend on the orientation of the integral on the boundary. We agree on the positive (counter-clockwise) orientation for this definition. 
\end{remark} 
\begin{remark}
 In addition, there is no upper limit to the self-perimeter for non-centrally symmetric sets. Indeed, if $p\in B$ is very close to $\partial B$, then evidently the self-perimeter of $B-\{p\}$ can be very large.
\end{remark} 

The following result was established in \cite{makeev2003}.
\begin{theorem}[Makeev \cite{makeev2003}]\label{thcon2D}
 Let $K \subset \mathbb{R}^2$ be a bounded convex set and $p\in Int(K)$. Then $p\rightarrow P_B(K-p)$ is strongly convex in $Int(K)$. In particular, there exists a unique interior point $p \in \text{int}(K)$ that minimizes the self-perimeter $P_B(K-\{p\})$.
\end{theorem}

For the sake of completeness, we provide a proof below.
The self-perimeter of $K$ with respect to an interior point $p$ is defined by the integral:
 $$P_B(K-p) = \int_{\partial K} \frac{1}{r_{K,p}(\vec{\tau}(s))} ds$$
where $r_{K,p}(v)$ is the distance from the point $p$ to the boundary $\partial K$ in the direction $v$, and $\vec{\tau}(s)$ is the tangent vector at the point $s \in \partial K$.

\textbf{Step 1: Concavity of the radius function}\\
For a fixed direction $v$, consider the function $p \mapsto r_{K,p}(v)$. Let $p_1, p_2 \in \text{int}(K)$ and $\lambda \in [0,1]$. Denote $p_\lambda = \lambda p_1 + (1-\lambda) p_2$. By the convexity of $K$, the line segment connecting the boundary points $p_1 + r_{K,p_1}(v)v$ and $p_2 + r_{K,p_2}(v)v$ is entirely contained within $K$. Therefore, the distance from $p_\lambda$ to the boundary in the direction $v$ satisfies:
 $$r_{K,p_\lambda}(v) \ge \lambda r_{K,p_1}(v) + (1-\lambda) r_{K,p_2}(v)$$
Hence, $r_{K,p}(v)$ is a concave function of $p$.

\textbf{Step 2: Strict convexity of the integrand}\\
The real-valued function $f(t) = \frac{1}{t}$ is decreasing and strictly convex for $t > 0$. The composition of a decreasing, strictly convex function with a concave function yields a strictly convex function. Therefore, for any given tangent direction, the function $p \mapsto \frac{1}{r_{K,p}(\vec{\tau}(s))}$ is strictly convex on the interior $\text{int}(K)$.

\textbf{Step 3: Convexity of the functional}\\
Since the self-perimeter $P_B(p)$ is an integral (a continuous sum) of strictly convex functions over the boundary $\partial K$ with respect to a positive length measure $ds$, the function $P_B(p)$ itself is strictly convex on $\text{int}(K)$.

\textbf{Step 4: Coercivity at the boundary}\\
As the point $p$ approaches the boundary $\partial K$, the distance to the boundary in the tangent directions (close to the point of tangency) tends to $0$. Consequently, the integrand approaches infinity, and thus $P_B(p) \to \infty$ as $p \to \partial K$.

\textbf{Conclusion}\\
A strictly convex function defined on a bounded open set $\text{int}(K)$, which approaches infinity on the boundary of the set, must attain a global minimum at a unique interior point. Therefore, there exists a unique point that minimizes the self-perimeter.

\subsection{Self-Perimeter of a Regular $k$-gon}
For regular $k$-gons, as well as for the family of planar convex bodies possessing $2\pi/k$ rotational symmetry, the self-perimeter exhibits periodic behavior depending on the remainder of $k$ modulo four. Let $K$ be a regular $k$-gon centered at the origin, with circumradius $R$. The Euclidean length of each edge is $L = 2R \sin(\pi/k)$. By symmetry, the self-perimeter is $k$ times the contribution of a single edge:
 $$P = k \frac{L}{r}$$
where $r$ is the distance from the origin to the boundary in the direction parallel to the edge. Assuming a vertex is at angle $0$, the first edge connects the vertices at angles $0$ and $2\pi/k$. The vector parallel to this edge has the angle $\alpha = \frac{\pi}{2} + \frac{\pi}{k}$. We cast a ray from the origin at angle $\alpha$. The boundary of $K$ consists of line segments with normal vectors at angles $\phi_m = \frac{(2m+1)\pi}{k}$ for integer $m$. The perpendicular distance from the origin to any edge is $d = R \cos(\pi/k)$. The ray intersects the $m$-th edge at a distance:
 $$r(\alpha) = \frac{R \cos(\pi/k)}{\cos(\alpha - \phi_m)}$$
	
The ray falls on the $m$-th edge if its angle $\alpha$ is within the sector $\left[\frac{2m\pi}{k}, \frac{2(m+1)\pi}{k}\right]$. Solving for $m$ yields $m = \lfloor \frac{k+2}{4} \rfloor$. The phase difference $\Delta = \alpha - \phi_m$ dictates the length of $r(\alpha)$ and strictly depends on $k \pmod 4$:
 $$\Delta = \frac{\pi}{k} \left( \frac{k}{2} - 2m \right)$$
	
Substituting the remainders of $k \pmod 4$ into $\Delta$, we compute $P = 2k \tan(\pi/k) \cos(\Delta)$. This dictates the nature of convergence to the limit of $2\pi$, leading to the following modular partition:
\begin{itemize}
    \item \textbf{Case $k \equiv 0 \pmod 4$}: $m = k/4 \implies \Delta = 0$. $P = 2k \tan(\pi/k)$. 
    \item \textbf{Case $k \equiv 1 \pmod 4$ and $k \equiv 3 \pmod 4$}: $|\Delta| = \frac{\pi}{2k}$. $P = 2k \tan(\pi/k) \cos\left(\frac{\pi}{2k}\right)$.
    \item \textbf{Case $k \equiv 2 \pmod 4$}: $m = (k+2)/4 \implies |\Delta| = \frac{\pi}{k}$. $P = 2k \sin(\pi/k)$. The regular polygon serves as a minimum.
\end{itemize}

\subsection{Busemann's definition}
In the case of non-centrally symmetric sets, Busemann's definition deviates from (\ref{per2}). Busemann considered the integral:
\be\label{Busdef} P_{Bus}(B)= \int_{\partial B} \frac{2ds}{r_B(\vec{\tau}(s))+r_B(-\vec{\tau}(s))} \ee
namely, the integral is taken over twice the inverse length of \textit{the chord} parallel to the tangent and passing through the origin. Note that $P_{Bus}(B)=P_B(B)$ for $B\in CCS$ and that in general, unlike $P$, $P_{Bus}$ is independent of the orientation of $\partial B$ (Remark \ref{rem3.1}). This definition is more natural when extending to higher dimensions, where the length of the chord is replaced by the Lebesgue measure of the parallel hyperplane (see Sec \ref{sec40}). Evidently, for CCS sets $P_{Bus}= P$ since $r_B(\phi)=r_B(-\phi)$. 

Theorem \ref{thcon2D} is valid also for $P_{Bus}$. The proof of this theorem can be extended to $P_{Bus}$ since $v\rightarrow r_{K, p}(-v)$ and $v\rightarrow r_{K, p}(v)+ r_{K, p}(-v)$ are concave as well. Surprisingly, we did not find an explicit formulation of this result in the literature, so we formulate it below:
\begin{theorem}\label{thcon2DBus}
 Let $K \subset \mathbb{R}^2$ be a bounded convex set and $p\in Int(K)$. Then $p\rightarrow P_{Bus}(K-p)$ is strongly convex in $Int(K)$. In particular, there exists a unique interior point $p \in \text{int}(K)$ that minimizes the Busemann self-perimeter $P_{Bus}(K-\{p\})$.
\end{theorem}
The following Theorem can also be easily extended to Busemann's self-perimeter:
\begin{theorem}[Grünbaum \& Martini \cite{grunbaum1966}]\label{th3.4}
 For any convex body $K \subset \mathbb{R}^2$, let $p^*(K)$ be its generalized centroid (the unique point minimizing the self-perimeter functional). Then
 \begin{equation}
 P_B(K-p^*(K)) = \min_{p \in \text{int}(K)} P_B(K-p) \le 9
 \end{equation}
\end{theorem}

\subsubsection*{Example}
Let $T \subset \mathbb{R}^2$ be a triangle with Euclidean edge lengths $L_1, L_2, L_3$ and corresponding vertices $V_1, V_2, V_3$. Let $p \in \text{int}(T)$ be an interior reference point. 
	
The point $p$ can be uniquely represented by its strictly positive barycentric coordinates $(\alpha, \beta, \gamma)$ such that:
 \begin{equation}
 p = \alpha V_1 + \beta V_2 + \gamma V_3 \quad \text{where} \quad \alpha + \beta + \gamma = 1
 \end{equation}
	
The asymmetric self-perimeter (\ref{per2}) relies on the lengths of the single directional rays $r_i$ originating from $p$ to the boundary in the directions parallel to the edges. For a triangle, these ray lengths evaluate to $r_1 = \alpha L_1$, $r_2 = \beta L_2$, and $r_3 = \gamma L_3$. The self-perimeter is the discrete sum over the edges:
 \begin{equation}
 P(B-p) = \sum_{i=1}^{3} \frac{L_i}{r_i} = \frac{L_1}{\alpha L_1} + \frac{L_2}{\beta L_2} + \frac{L_3}{\gamma L_3}
 \end{equation}
Which simplifies entirely to a function of the barycentric coordinates:
 \begin{equation}
 P(B-p)=\frac{1}{\alpha(p)} + \frac{1}{\beta(p) }+ \frac{1}{\gamma(p)}
 \end{equation}
Note that this perimeter is independent of the orientation of the triangle since the opposite orientation just permutes the barycentric coordinates but preserves the sum.

The Busemann perimeter (\ref{Busdef}) normalizes the Euclidean surface element using the full length of the parallel chords $c_i$ passing through $p$, rather than single rays. The full chord length parallel to an edge is the sum of the forward ray and the backward ray: $c_i = r_i + r_i^-$. For a 2D triangle, the Busemann definition multiplies the sum by $\omega_1 = 2$. The chords parallel to the edges evaluate to $c_1 = (1-\alpha)L_1$, $c_2 = (1-\beta)L_2$, and $c_3 = (1-\gamma)L_3$. The Busemann self-perimeter is given by:
 \begin{equation}
 P_{Bus}(B-p) = \sum_{i=1}^{3} \frac{2 L_i}{c_i} = \frac{2 L_1}{(1-\alpha)L_1} + \frac{2 L_2}{(1-\beta)L_2} + \frac{2 L_3}{(1-\gamma)L_3}
 \end{equation}
Which simplifies to:
 \begin{equation}
 P_{Bus}(B-p)=2 \left( \frac{1}{1-\alpha(p)} + \frac{1}{1-\beta(p)} + \frac{1}{1-\gamma(p)} \right)
 \end{equation}
Since $\alpha+\beta+\gamma=1$, we can write:
 $$ P_{Bus} = 2 \left( \frac{1}{\beta(p)+\gamma(p)} + \frac{1}{\alpha(p)+\gamma(p)} + \frac{1}{\alpha(p)+\beta(p)} \right) \ . $$
By the inequality $a^{-1}+b^{-1}\geq 4(a+b)^{-1}$, we obtain:
 $$9\leq P_{Bus}(T) \leq P_B(T) $$
where both equalities hold whenever $p$ is the centroid $\alpha=\beta=\gamma=1/3$.

\section{Self-perimeter in $\R^n$} \label{sec40} 
We define the affine invariant perimeter $P_B(K)$ of a set $K$ in a normed space $(\R^n, \|\cdot\|_B)$ where $B\subset \R^n$ is a convex, centrally symmetric set, by:
$$ P_B(K):= \int_{\partial K} \frac{\omega^{(n-1)}_B(\nu_x) }{{\cal H}_{n-1}(B\cap \Sigma^{n-1}(\nu_x))}d{\cal H}_{n-1}(x)$$ 
where $\Sigma^{n-1}(\nu_x)$ is the $(n-1)$-subspace perpendicular to the normal at $x\in\partial K$  to the normal $\nu_x$, and $\omega_B^{(n-1)}$ is defined recursively via:
 \begin{defi}\label{def4.1}
  Given an orthonormal frame $\nu_{k+1}, \ldots \nu_n$ in $\R^n$, let $B^{(k)}(\nu_{k+1}, \ldots, \nu_n)\subset\R^k$ be the intersection of $B$ with the subspace orthogonal to $\text{Sp}(\nu_{k+1}, \ldots, \nu_n)$. Define:
   \be\label{peri} \omega_B^{(k)}(\nu_{k+1}, \ldots \nu_n) = k^{-1}\int_{\partial B^{(k)}(\nu_{k+1}, \ldots , \nu_n)} \frac{ \omega_B^{(k-1)}(\nu_x, \nu_{k+1}, \ldots , \nu_n)}{{\cal H}_{k-1}(B^{(k-1)}(\nu_x, \nu_{k+1}, \ldots , \nu_n) )}d{\cal H}_{k-1}(x)\ee
   for $k=2, \ldots, n-1$, where $\nu_x\in \text{Sp}^\perp(\nu_{k+1}, \ldots, \nu_n)$ is the normal to $\partial B^{(k)}(\nu_{k+1}, \ldots , \nu_n)$ at $x$. This recursive definition starts at $k=2$ with $\omega^{(1)}(\nu_2, \ldots, \nu_n)\equiv 2$, independent of the frame $\nu_2, \ldots, \nu_n$.
 \end{defi}
 
 In particular:
 \be\label{n=2} \omega_B^{(n)} = n^{-1}\int_{\partial B^{(n)} }\frac{ \omega^{(n-1)}_B(\nu_x) d{\cal H}_{n-1}(x)}{{\cal H}_{n-1}(B^{(n-1)}(\nu_x)) } \ee
is the \textit{self-volume} of the unit ball $B^{(n)}$ while $P(B^{(n)})=n\omega^{(n)}_B$ is its self-perimeter. 

 \begin{remark}
  Note that for $n=2$, (\ref{n=2}) is consistent with (\ref{per2}) since $\omega^{(1)}=2$. Indeed, if $n=2$, then ${\cal H}_1(B^{(1)}(\phi))=2r_B(\phi)$ because $B$ is centrally symmetric. However, if $B$ is not centrally symmetric, then (\ref{n=2}) deviates from (\ref{per2}). Altogether, Definition \ref{def4.1} is valid for \textit{any} convex set $B$ provided $\{0\}\in Int(B)$.
 \end{remark}
 
 \subsection{Examples}\label{exam}
 
 \subsubsection*{The hypercube $[-1,1]^n$}
 The central cross-section of $B_n=[-1,1]^n$ passing through the origin and orthogonal to its normal is $B_{n-1}=[-1,1]^{n-1}$. The Euclidean $(n-1)$-dimensional volume of this cross-section is $\mathcal{H}_{n-1}(B_{n-1}) = 2^{n-1}$. Because the normal vector is constant across a single face, the integrand is also constant. The integral over the boundary $\partial B_n$ is the sum of the integrals over the $2n$ faces:
 $$\omega^{(n)}_{B_n} =n^{-1} \sum_{k=1}^{2n} \int_{\text{face}_k} \frac{\omega^{(n-1)}_{B_{n-1}}(\nu_k)}{2^{n-1}} d\mathcal{H}_{n-1}(x)$$
 where $\nu_k$ is the constant unit vector perpendicular to the $k$-th face. By Theorem \ref{thinlin}, it follows that $\omega^{(n-1)}_{B_{n-1}}(\nu_k)$ is independent of $k$, hence $\omega^{(n-1)}_{B_{n-1}}(\nu_k) =\omega^{(n-1)}_{B_{n-1}}$. In addition, the Lebesgue measure of $B_{n-1}$ is $2^{n-1}$, so:
  $$\omega^{(n)}_{B_n} =n^{-1}\times 2n\times \omega^{(n-1)}_{B_{n-1}} = 2 \omega^{(n-1)}_{B_{n-1}} $$
 Since $\omega^{(1)}_{B_1}=2$, it follows:
 $$ \omega^{(n)}_{B_n}= 2^n$$
in agreement with the Lebesgue measure of $[-1,1]^n$.

 \subsubsection*{The Euclidean ball} 
 Since the Lebesgue measures of the cross-sections of the Euclidean ball are all constants equal to $\omega^{(n-1)}$, it follows that the self-perimeter is:
 $$\omega_B^{(n)}=n^{-1}\int_{{\mathbb{S}^{n-1}}} \frac{\omega_B^{(n-1)} (\nu_x)}{\omega^{(n-1)}} d{\cal H}_{n-1}(x) $$
 Since $\int_{{\mathbb{S}^{n-1}}}d{\cal H}_{n-1} =n\omega^{(n)}$ and $\omega_B^{(n-1)}$ is independent of orientation, we obtain:
  $$\omega_B^{(n)}=\frac{\omega_B^{(n-1)}\omega^{(n)}} {\omega^{(n-1)}} \ . $$
 Since $\omega^{(1)}=\omega_B^{(1)}=2$, we obtain the equality $\omega_B^{(n)}=\omega^{(n)}$ for any $n\geq 2$.

 \subsubsection*{The $n$-simplex}
In this example, we study the self-volume of a non-centrally symmetric ball. Let $\Delta^n$ be an $n$-dimensional simplex with $n+1$ vertices $v_1, \ldots, v_{n+1}\in\R^n$. Any interior point $p \in \text{int}(\Delta^n)$ can be uniquely represented by its strictly positive barycentric coordinates $\lambda = (\lambda_1, \dots, \lambda_{n+1})$, where $p=\sum_{i=1}^{n+1} \lambda_i v_i$ and $\sum_{i=1}^{n+1} \lambda_i = 1$, with $\lambda_i > 0$ for all $i$. Let $F_i$ denote the $(n-1)$-dimensional face opposite to the $i$-th vertex.

 \begin{theorem} \label{simn} 
    The self-volume of the simplex $\Delta^{(n)}$ with center at $p$ is:
 $$ \omega_{\Delta^n}^{(n)} = \frac{2}{n!} \sum_{k_1, \dots, k_{n-1}} \prod_{m=1}^{n-1} \frac{1}{1 - \sum_{r=1}^{m} \lambda_{k_r}} $$
 where $\{k_1, \ldots, k_{n-1}\}\subset \{1, \ldots, n+1\}$ are $n-1$ different indices ($k_i\neq k_j$ for $i\neq j$) and $\lambda_1, \ldots, \lambda_{n+1}$ are the barycentric coordinates of $p$.
 \end{theorem}
\begin{proof}
By applying the recursive definition of $\omega_{\Delta^n}$ (\ref{peri}), the integral over the boundary $\partial \Delta^n$ decomposes into a sum over the $n+1$ faces. For a specific face $F_i$, the integration is evaluated over the $(n-1)$-dimensional cross-section $S_i$ passing through $p$ and parallel to $F_i$. Because $S_i$ is parallel to $F_i$, it is a simplex similar to $F_i$ with a linear scaling factor of $(1-\lambda_i)$. Consequently, the ratio of their Euclidean volumes, which appears in the denominator, is given by:
\begin{equation}
\frac{\mathcal{H}_{n-1}(S_i)}{\mathcal{H}_{n-1}(F_i)} = (1-\lambda_i)^{n-1}
\end{equation}
 
 \textbf{Base Case ($n=1$):} \\
 For $n=1$, the sum of indices cancels out and the empty product is defined as 1. We obtain:
 $$ \omega_{\Delta^1}^{(1)}= \frac{2}{1!} = 2 $$
 
 \textbf{Inductive Step:} \\
 Assume the formula holds for $n$:
 $$ \omega_{\Delta^n}^{(n)} = \frac{2}{n!} \sum_{k_1, \dots, k_{n-1}} \prod_{m=1}^{n-1} \frac{1}{1 - \sum_{r=1}^{m} \lambda_{k_r}} $$
 
 We will prove its validity for $n+1$. Based on the decomposition properties of the volume, the following recursive formula holds:
 $$ \omega_{\Delta^{n+1}}^{(n+1)} = \frac{1}{n+1} \sum_{k_n} \frac{1}{1 - \sum_{r=1}^{n} \lambda_{k_r}} \omega^{(n)}_{\Delta^{n}} $$
 
 Substitute the expression for $\omega_{\Delta^n}^{(n)}$ from the inductive hypothesis into the recursion:
 $$ \omega_{\Delta^{n+1}}^{(n+1)} = \frac{1}{n+1} \sum_{k_n} \frac{1}{1 - \sum_{r=1}^{n} \lambda_{k_r}} \left( \frac{2}{n!} \sum_{k_1, \dots, k_{n-1}} \prod_{m=1}^{n-1} \frac{1}{1 - \sum_{r=1}^{m} \lambda_{k_r}} \right) $$
 $$= \frac{2}{(n+1)!} \sum_{k_1, \dots, k_n} \prod_{m=1}^{n} \frac{1}{1 - \sum_{r=1}^{m} \lambda_{k_r}} \ . $$
\end{proof}

\subsection{Convexity of self-volume}
By Theorem \ref{simn} we obtain:
 \begin{cor}\label{cor4.1}
 The self-volume $\omega^{(n)}(\Delta^n-p)$ is a strictly convex function of the center $p$. In particular, there exists a unique $p^*\in \Delta^n$ for which $\omega^{(n)}(\Delta^n-p^*)$ attains its minimum, and $\lambda_i=\frac{1}{n+1}$, $i=1, \ldots, n+1$ are the barycentric coordinates of $p^*$.
 $$ \min_p\omega_{\Delta^n}^{(n)}= \frac{(n+1)^n}{n!} \ . $$
\end{cor}
\begin{proof}
 Strict convexity follows from the log-convexity of $\log(f)$ where $f(\vec{\lambda}) = \prod_{m=1}^{n-1} \frac{1}{1-S_m} $ and $S_m=\sum_{i=1}^m \lambda_{k_i}$. Hence, the minimum is unique. Since this function is invariant under permutation of $\lambda_i$, the minimum must attain where all $\lambda_i$ are equal, namely $\lambda_i=(n+1)^{-1}$.
\end{proof}

\begin{theorem}\label{product}
  Let $A \subset \mathbb{R}^n$ and $B \subset \mathbb{R}^m$ be convex sets. The self-volume of their Cartesian product $A \times B \subset \mathbb{R}^{n+m}$ satisfies:
  \begin{equation}
  \omega^{(n+m)}_{A \times B} = \omega^{(n)}_A \omega^{(m)}_B
  \end{equation}
\end{theorem}
\subsubsection*{Proof of Theorem \ref{product}}

We proceed by mathematical induction on the sum of the dimensions, $k = n + m$.

\paragraph{Base Case ($k = 2$):}
Let $n = 1$ and $m = 1$. Both $A$ and $B$ are one-dimensional segments, denoted $I_1$ and $I_2$. The Cartesian product $I_1 \times I_2$ is a 2D rectangle. The boundary $\partial(I_1 \times I_2)$ decomposes into two parts: $(\partial I_1 \times I_2) \cup (I_1 \times \partial I_2)$. For the faces $\partial I_1 \times I_2$, the normal vectors are parallel to $I_1$. The orthogonal cross-section is $I_2$. The Euclidean length of these two faces is $2|I_2|$.
The contribution to the self-perimeter is:
\[
\int_{\partial I_1 \times I_2} \frac{\omega^{(1)}_{I_2}}{|I_2|} d\mathcal{H}_1 = 2|I_2| \frac{\omega^{(1)}_{I_2}}{|I_2|} = 2\omega^{(1)}_{I_2}
\]
By symmetry, the integral over $I_1 \times \partial I_2$ evaluates to $2\omega^{(1)}_{I_1}$. Thus, the total self-perimeter is:
\[
\omega^{(2)}_{I_1 \times I_2}= \omega^{(1)}_{I_2}+ \omega^{(1)}_{I_1}=4
\]
This is consistent with the self-volume of $[-1,1]^2$ via Section \ref{exam}.

\paragraph{Inductive Hypothesis:}
Assume the theorem holds for any pair of convex sets whose dimensions sum to $k-1$. That is, for any integers $p, q \ge 1$ such that $p+q = k-1$, and appropriate convex sets $U \subset \mathbb{R}^p, W \subset \mathbb{R}^q$, we have:
\[
\omega^{(k-1)}_{U \times W}= \omega^{(p)}_U \omega^{(q)}_W
\]

\paragraph{Inductive Step:}
Consider convex sets $A \subset \mathbb{R}^n$ and $B \subset \mathbb{R}^m$ such that $n+m = k$. The essential boundary of the Cartesian product decomposes into two disjoint components (up to a set of measure zero):
\[
\partial(A \times B) = (\partial A \times B) \cup (A \times \partial B)
\]
We evaluate the self-perimeter integral over these two components separately. For any point $(x, y) \in \partial A \times B$, the outward normal vector $\nu_x$ lies entirely within the subspace $\mathbb{R}^n$. The orthogonal cross-section is given by $A_x \times B$, where $A_x = A \cap \Sigma(\nu_x)$ is an $(n-1)$-dimensional convex set. The sum of the dimensions of $A_x$ and $B$ is $(n-1) + m = k-1$. By the inductive hypothesis, we can decompose its volume:
\[
\omega^{(k-1)}_{A_x \times B}= \omega^{(n-1)}_{A_x}\omega^{(m)}_B
\]
The contribution to the self-perimeter from this component is:
\begin{align*}
I_1 &= \int_{\partial A \times B} \frac{\omega^{(k-1)}_{A_x \times B}}{\mathcal{H}_{k-1}(A_x \times B)} d\mathcal{H}_{k-1} \\
&= \int_{\partial A} \int_{B} \frac{\omega^{(n-1)}_{A_x} \omega^{(m)}_B}{|A_x| |B|} dy \, d\mathcal{H}_{n-1}(x)
\end{align*}
Integration over $y \in B$ yields the Euclidean volume $|B|$, which completely cancels the $|B|$ in the denominator:
\begin{align*}
I_1 &= \omega^{(m)}_B \int_{\partial A} \frac{\omega^{(n-1)}_{A_x}}{|A_x|} d\mathcal{H}_{n-1}(x)
\end{align*}
By the recursive definition, the remaining integral is exactly the self-perimeter of $A$, denoted $P^{(n)}(A)$. Since $P^{(n)}(A) = n \omega^{(n)}_A$, we obtain:
\[
I_1 = n \omega^{(n)}_A \omega^{(m)}_B
\]
Similarly:
\[
\omega^{(k-1)}_{A \times B_y} = \omega^{(n)}_A\omega^{(m-1)}_{B_y}
\]
The contribution to the self-perimeter from this part is:
$$
I_2 = \int_{A \times \partial B} \frac{\omega^{(n)}_A \omega^{(m-1)}_{B_y}}{|A| |B_y|} dx \, d\mathcal{H}_{m-1}(y) = m \omega^{(n)}_A \omega^{(m)}_B
$$
So:
$$\omega^{(n+m)}_{A\times B}= (m+n)^{-1}(I_1+I_2) = \omega^{(n)}_A \omega^{(m)}_B \ . $$
\qed

\subsection{Strict Convexity for Cartesian Products}\label{sec4.3}

From Theorem \ref{product} and Theorem \ref{thcon2DBus} we obtain:
\begin{cor}
 If $A\in \R^n$ is a Cartesian product of convex sets of dimension 2, intervals of dimension 1, and simplices of any dimension, then $p\rightarrow \omega^{(n)}(A-\{p\})$ is a strongly convex function in $Int(A)$. In particular, there exists a unique $p\in Int(A)$ which minimizes $\omega^{(n)}(A-\{p\})$.
\end{cor}

\subsection{Conjecture} \label{subsec4.3} 
\begin{itemize}
    \item The CCS (convex centrally symmetric) body in $\R^n$ of maximal self-volume is the $n$-cube $\omega^{(n)}_B=2^n$.
    \item The non-centrally symmetric, convex body in $\R^n$ of maximal self-volume with respect to its centroid is the $n$-simplex $B=\Delta^n$ ($\omega^{(n)}_B = \frac{(n+1)^n}{n!} $).
    \item If $n=2k$, then the CCS body of minimal self-volume is the Cartesian product of $n$ hexagons ($\omega^{(n)}_B=3^k$). If $n=2k+1$, it is the Cartesian product of $k$ hexagons and a 1D interval ($\omega^{(n)}_B=2\times 3^k$).
\end{itemize} 

\subsection{Invariance of the Self-Perimeter Under Affine Transformations}


\subsubsection*{Proof of Affine Invariance}

We rely on the following geometric property established by Busemann, which we quote without proof:

\begin{theorem}\label{Busth} [Busemann] The ratio of the Euclidean surface measure to the volume of the parallel central cross-section is invariant under any non-singular linear transformation $A: \mathbb{R}^k \to \mathbb{R}^k$. Specifically, both the surface element $d\mathcal{H}^{k-1}(x)$ and the cross-sectional volume ${\cal H}_{k-1}(K \cap \tau_x^\perp)$ scale by the identical Jacobian factor $J = |\det A| \, |A^{-t} n_x|$, leading to exact cancellation:
\end{theorem}
\begin{equation}
\frac{d\mathcal{H}_{k-1}(Ax)}{{\cal H}_{k-1}(A(K) \cap A(\tau_x)^\perp)} = \frac{d\mathcal{H}^{k-1}(x)}{{\cal H}_{k-1}(K \cap \tau_x^\perp)}
\end{equation}

\begin{theorem}\label{thinlin} For any $n \ge 1$ and any non-singular affine transformation $T$ preserving the origin, the self-volume $\omega^{(n)}_B$ is invariant: $\omega^{(n)}(B)= \omega^{(n)}(T(B))$.
\end{theorem}

The proof of Theorem \ref{thinlin} follows from Busemann's Theorem \ref{Busth} and an induction argument based on the recursive Definition \ref{def4.1}.

\section{Alexandrov-type problem for self-surface area}\label{sec6}
\subsection{Background and Motivation}
In classical convex geometry, the Minkowski and Alexandrov problems seek to reconstruct a convex body from its standard Euclidean surface area measure or integral Gauss curvature, respectively. In this framework, we introduce a novel variation: an Alexandrov-type problem formulated with respect to the \textit{self-surface measure} of a centrally symmetric convex set. Instead of relying on the Euclidean surface area, the perimeter and volume are defined intrinsically. For a centrally symmetric convex set $K \subset \mathbb{R}^n$, the self-perimeter is constructed recursively by integrating over the essential boundary, normalized by the self-volume of the $(n-1)$-dimensional cross-sections. The inverse problem therefore asks: Given a measure $\mu$ on $S^{n-1}$, does there exist a centrally symmetric convex set $B \subset \mathbb{R}^n$ whose self-surface measure is exactly $ \mu$?

\begin{defi}
 Let $\mu$ be a measure on $S^{n-1}$. For a centrally symmetric convex set $B$, let $T_B:S^{n-1}\rightarrow \partial B$ be given by $T_B(\theta)=r_B(\theta)\theta$ (see Definition \ref{def1}). Let $\nu_B(\theta)$ be the exterior normal to $\partial B$ at the point $x= T_B(\theta)$. Using the Jacobian of the radial map $d{\cal H}_{n-1}(x)= \frac{r_B^{n-1}(\theta)}{<\theta, \nu_B(\theta)>} d\theta$, we may convert (\ref{n=2}) to:
   $$ \omega^{(n)}_B:=n^{-1}\int_{S^{(n-1)}}\frac{\omega^{(n-1)}(B^{(n-1)}(\nu_B(\theta)))}{{\cal H}_{n-1}(B^{(n-1)}(\nu_B(\theta))))} \frac{r_B^{n-1}(\theta)}{<\theta, \nu_B(\theta)>} d\theta \ . $$
  
  {\bf Alexandrov-type problem:}
  Let $\mu$ be a positive measure on $S^{n-1}$. Find a centrally symmetric convex set $B$ in $\R^n$ whose self-surface measure is consistent with $\mu$, namely:
  \be\label{newf}\mu(d\theta)= \frac{\omega^{(n-1)}(B^{(n-1)}(\nu_B(\theta)))}{{\cal H}_{n-1}(B^{(n-1)}(\nu_B(\theta))))} \frac{r_B^{n-1}(\theta)}{<\theta, \nu_B(\theta)>} d\theta \ . \ee
\end{defi}
 
 \vskip .4in
 
Using (\ref{1}), we pose the Alexandrov-type problem in $\R^2$: Let $\mu$ be a $\pi$-periodic density on $S^1$. Is there a central convex set $B$ such that:
 \be\label{Apro} \mu(d\theta)= \frac{\sqrt{r^2_B(\theta)+ (r^{'}_B(\theta))^2}d\theta}{r_B\left( \theta + \arctan\left(\frac{r_B(\theta)}{r_B^{'}(\theta)}\right)\right)} \ \ \ee
 
Let $\mu:=e^{\epsilon(\phi+\phi_0)} d\theta$ where $\phi_0\in\R$ and $\int_0^{2\pi} \phi d\theta=0$. We may write $\phi = \phi^{(p)}+ \phi^{(u)}$ where:
  $$\phi^{(p)}= \sum_{k\in 4\mathbb{Z}-\{0\}} \phi_ke^{ik\theta} \ ; \ \ \phi^{(u)}= \sum_{k\not\in 4\mathbb{Z}} \phi_ke^{ik\theta} \ \ ; \phi_k=\phi_{-k}^* $$
\begin{lemma}
 There exists a unique $\phi_0\in\R$ such that:
 \be\label{phi0}
 \int_{\{s\in \mathbb{S}^1, \phi^{(p)}(s)+\phi_0>0\}} \sqrt{\phi^{(p)}(s)+\phi_0} \, ds - \int_{\{s\in \mathbb{S}^1, \phi^{(p)}(s)+\phi_0<0\}} \sqrt{|\phi^{(p)}(s)+\phi_0|} \, ds = 0.
 \ee
\end{lemma}
\begin{proof}
  Consider \[\Gamma(\gamma):= \int_{\{s\in \mathbb{S}^1 , \phi^{(p)}(s)-\gamma>0\}} \sqrt{\phi^{(p)}(s)-\gamma} \, ds - \int_{\{ s\in\mathbb{S}^1, \phi^{(p)}(s)-\gamma<0\}} \sqrt{-(\phi^{(p)}(s)-\gamma)} \, ds \]
  It follows that $\Gamma$ is continuous, monotone decreasing, $\Gamma$ is negative for $\gamma>>1$ and positive for $\gamma<<1$, so the lemma is obtained by the mean value theorem.
\end{proof}
\begin{theorem}\label{th6.1}
  If $\phi^{(p)}\not\equiv 0$ and $\phi_0$ satisfies (\ref{phi0}), then the asymptotic expansion of the Alexandrov problem around the unit circle $r_B\equiv 1$ is:
  $$ r_B^{'}= \pm \eps^{1/2}\operatorname{sgn}(\phi^{(p)}+\phi_0)\sqrt{\frac{2}{3}|\phi^{(p)}+\phi_0|} +O(\eps)$$
\end{theorem}
  
\begin{proof} 
Let $\zeta =\ln (r_B)$ where $r_B$ is a solution. Taking the log of both sides we get:
 \begin{equation}
 \phi(\theta) +\phi_0= \zeta(\theta) + \frac{1}{2}\ln(1 + (\zeta'(\theta))^2) - \zeta\left(\theta + \arctan\left(\frac{1}{\zeta'(\theta)}\right)\right)
 \end{equation}
 
 We expand the argument for $\zeta=o(1)$ using $\arctan(1/\zeta') = \text{arccot}(\zeta') \approx \pi/2 - \zeta'$.
 \begin{equation}\label{27}
 \epsilon( \phi(\theta) +\phi_0)= \zeta(\theta) - \zeta(\theta + \pi/2) + \zeta'(\theta)\zeta'(\theta + \pi/2) + \frac{1}{2}(\zeta'(\theta))^2 + \mathcal{O}(\zeta^3)
 \end{equation}
 
 Let $L(\zeta)_{(\theta)}:= \zeta(\theta) -\zeta(\theta+\pi/2)$. The kernel of $L$ is the set of $\pi/2$-periodic functions. Its spectrum is $\lambda_k=(1-e^{ik\pi/2})$ and the eigenfunctions are just $e^{ik\theta}$, $k\in \mathbb{Z}$. To balance the equation, we perform an asymptotic expansion $\zeta = \epsilon^{1/2} \zeta_1 + \epsilon \zeta_2 + o(\epsilon)$. At $\mathcal{O}(\epsilon^{1/2})$, $L \zeta_1=0$. This implies $\zeta_1$ is purely $\pi/2$-periodic, containing only harmonics where $k \in 4\mathbb{Z}$. It implies $\zeta_1^{'}(\theta+\pi/2)=\zeta_1^{'}(\theta)$ so by (\ref{27}):
 \be\label{eqh} L \zeta_2=\phi+\phi_0 -\frac{3}{2} (\zeta_1^{'})^2 \ . \ee

 The right side of (\ref{eqh}) should contain only the harmonics with $k\not\in 4\mathbb{Z}$. Thus:
 $$ (\zeta_1^{'})^2= \frac{2}{3} (\phi^{(p)}+\phi_0)\ \rightarrow \zeta_1^{'}=\pm \operatorname{sgn} (\phi^{(p)}+\phi_0) \sqrt{ \frac{2}{3} |\phi^{(p)}+\phi_0|} $$
 Since $\zeta_1$ is $2\pi$-periodic it implies $\int_0^{2\pi} \zeta_1^{'}(\theta)d\theta=0$. This is satisfied iff $\phi_0$ satisfies (\ref{phi0}). 
\end{proof} 
\begin{remark} 
By (\ref{eqh}) we also get:
 $$ \zeta_2(\theta)=\sum_{k\not\in 4\mathbb{Z}} \frac{\phi_k e^{ik\theta}}{1-e^{ik\pi/2}} \mod \text{Ker}(L)$$
\end{remark} 
 
\section{Conclusion} 

In this paper, we introduced a natural, recursive definition for volume and surface area in $n$-dimensional normed spaces which preserves   the Euclidean ratio $P_B(B)/\omega^{(n)}_B(B)=n$ of perimeter to volume of the unit ball $B$.   We demonstrated that this construction is invariant under origin-preserving affine transformations, satisfying fundamental geometric requirements for measures in Minkowski spaces \cite{alvarez1998volumes, thompson1996minkowski}. 

We extend this definition also for non centrally symmetric convex sets $B\subset\R^n$ and prove that the self volume of a Cartesian products of sets is the product of their self-volume. As a result we verified  that the function $p \mapsto V(B-p)$  is convex for a class of convex bodies defined by a Cartesian products of simplexes, 2D convex sets and 1-D intervals. In particular we obtain the existence of an optimal center for such sets for which the self-volume is minimal.

Finally, we explore an Alexandrov-type problem for the associated self-surface measure \cite{schneider2014convex}. Through a perturbative analysis around the Euclidean disc, we identified a rigidity result for the surface measure $\mu$, revealing a significant obstruction to solvability: a necessary condition for the existence of a solution is the absence of certain harmonic components in the density perturbation. These findings open new avenues for research, particularly regarding the global solvability of the inverse problem and the extremal values of the self-volume in higher dimensions.

\section*{Acknowledgments}
The author wishes to acknowledge Jenni (also known as Gemini) for her invaluable assistance in editing, proofreading, and structuring this manuscript.

\end{document}